\documentclass{amsart}

\usepackage{jk_basic}

\usepackage{stmaryrd}

\setenumerate{label={\normalfont(\textit{\roman*})}, ref={(\textrm{\roman*})},
wide
}

\usepackage[mathscr]{euscript}

\DeclareMathAlphabet{\mathpzc}{OT1}{pzc}{m}{it}

\newcommand{\Ball}{\mathscr{B}}
\newcommand{\GP}{\operatorname{GP}}

\definecolor{fresh}{HTML}{1e5e06}
\definecolor{checked}{HTML}{02087F}
\definecolor{final}{HTML}{A66911}
\definecolor{wrong}{HTML}{c90000}
\definecolor{skip}{HTML}{ffffff}
\definecolor{normal}{HTML}{000000}

\definecolor{fresh}{HTML}{000000}\definecolor{checked}{HTML}{000000}
\definecolor{external}{HTML}{000000}\definecolor{fresh}{HTML}{000000}
\definecolor{minor-rev}{HTML}{000000}\definecolor{major-rev}{HTML}{000000}\definecolor{skip}{HTML}{000000}\definecolor{normal}{HTML}{000000}

\begin{document}

\author[J.\ Konieczny]{Jakub Konieczny}
\address{Department of Computer Science, University of Oxford,
Wolfson Building, Parks Road, Oxford OX1 3QD, UK}
\email{jakub.konieczny@gmail.com}

\title[Multiplicative generalised polynomials]{Multiplicative generalised polynomial sequences}
\date{\today}
 
\begin{abstract}
	We fully classify completely multiplicative sequences which are given by generalised polynomial formulae, and obtain a similar result for (not necessarily completely) multiplicative sequences under the additional restriction that the sequence is not zero almost everywhere. 
\end{abstract}

 
\maketitle

\section{Introduction}\color{checked}
\label{sec:intro}

\subsection{Background}
Generalised polynomials are expressions built up from polynomials with the use of the integer part function, addition and multiplication. Unlike classical polynomials, their generalised counterparts can be bounded or even finitely-valued without being constant. A notable example of this behaviour is provided by Sturmian sequences $\brabig{ \ip{\a (n+1)+\b} - \ip{\a n + \b}}_n$ ($\a \in (0,1)\setminus \QQ$, $\b \in [0,1)$).  Properties of bounded generalised polynomials have been extensively studied, see e.g.\ \cite{Haland-1993}, \cite{Haland-1994}, \cite{BergelsonLeibman-2007}, \cite{Leibman-2012}, \cite{AdamczewskiKonieczny-2023-TAMS} and references therein.

Given a sequence $f \colon \NN \to \RR$ it is often a non-trivial task to determine if $f$ is a generalised polynomial\footnote{Here and elsewhere, by a slight abuse of notation, we identify a generalised polynomial with the corresponding sequence.}. For instance, letting $f = 1_E$ be the indicator function of a set $E \subset \NN$ (given by $1_E(n) = 1$ if $n \in E$ and $1_E(n) = 0$ otherwise) we know that $f$ is generalised polynomial if $E = \{1,2,3,5,8,13\dots\}$ is the set of Fibonacci numbers, or if $E = \{1, 2, 4, 7, 13, 24,\dots\}$ is the set of Tribonacci numbers \cite{ByszewskiKonieczny-2018-TAMS}, or if $E$ is a subset one of the two aforementioned sets \cite{ByszewskiKonieczny-2023+}, but not if $E = \{1,2,4,8,16,32,\dots\}$ is the set of powers of $2$ \cite{Konieczny-2022-JLMS}.

With the above discussion in mind, given a family $\cF$ of sequences $f \colon \NN \to \RR$, it is natural to ask which elements of $\cF$ are generalised polynomials. More generally, we can ask the same question for complex-valued sequences, where we say that a sequence $f \colon \NN \to \CC$ is a generalised polynomial if it takes the form $f = f_0 + i f_1$ where $f_0,f_1 \colon \NN \to \RR$ are generalised polynomials.

\subsection{New results}
The purpose of this paper is to resolve the aforementioned question partially for multiplicative sequences, and fully for completely multiplicative sequences. Recall that a sequence $f \colon \NN \to \CC$ is multiplicative if $f(nm) = f(n)f(m)$ for all $n,m \in \NN$ with $\gcd(n,m) = 1$, and completely multiplicative if the requirement $\gcd(n,m) = 1$ can be removed. We will say that a statement $\phi(n)$ holds for asymptotically almost all $n \in \NN$ if the set of $n$ for which it is false has asymptotic density zero, i.e., $\#\set{n \leq N}{\neg \phi(n)}/N \to 0$ as $N \to \infty$.

\begin{alphatheorem}\label{thm:intro:main}
	Let $f \colon \NN \to \CC$ be a multiplicative generalised polynomial sequence. Then either there exists a periodic multiplicative sequence $\chi \colon \NN \to \CC$ and an exponent $a \in \NN_0$ such that $f(n) = \chi(n)n^a$ for all $n \in \NN$, or $f(n) = 0$ for asymptotically almost all $n \in \NN$.
\end{alphatheorem}

\begin{remark}\label{rmk:intro:main}
	Periodic multiplicative sequences were classified in \cite{LeitmannWolke-1976}. While the classification is not particularly difficult, it is rather technical to formulate, so we do not reproduce it here. Combined with this result, Theorem \ref{thm:intro:main} gives a complete description of multiplicative generalised polynomials which are not almost everywhere zero.
\end{remark}

Existence of non-trivial examples of multiplicative (but not completely multiplicative) generalised polynomial sequences $f \colon \NN \to \CC$ such that $f(n) = 0$ for asymptotically almost all $n \in \NN$ is a straightforward corollary of the fact that each sufficiently sparse sequence is a generalised polynomial, stated more formally below.

\begin{theorem}[{\cite[Thm.\ C]{ByszewskiKonieczny-2018-TAMS}}]\label{thm:sparse}
	Let $c > 0$ and let $E = \set{n_{i}}{i \in \NN} \subset \NN$ be a set satisfying $n_{i+1} > n_i^{1+c}$ for all $i \in \NN$. Then $1_E$ is a generalised polynomial.
\end{theorem}

\begin{example}\label{rmk:intro:main-2}
	Let $c > 0$ and let $A = \set{\a_i}{i \in \NN} \subset \NN_0$ be a set satisfying $\a_{i+1} > (1+c) \a_i$ for all $i \in \NN_0$. Let $f \colon \NN \to \{0,1\}$ be a sequence given by
\[
	f(n) = 
	\begin{cases}
		1 &\text{if } n = 2^\a \text{ for some } \a \in A,\\
		0 &\text{otherwise}.
	\end{cases}
\]
	Then $f$ is multiplicative by direct inspection, and $f$ is a generalised polynomial by Theorem \ref{thm:sparse}.
\end{example}

In the case of completely multiplicative sequences, we obtain a full classification by combining Theorem \ref{thm:intro:main} with the fact that there are no non-trivial generalised polynomials that are simultaneously invariant under a dilation and almost everywhere zero.

\begin{theorem}[{Corollary of \cite[Thm.\ A]{Konieczny-2022-JLMS}}]\label{thm:times-k}
Let $k \geq 2$ be an integer and let $f \colon \NN \to \{0,1\}$ be a generalised polynomial such that $f(kn) = f(n)$ for all $n \in \NN$ and $f(n) = 0$ for asymptotically almost all $n \in \NN$. Then $f(n) = 0$ for all $n \in \NN$.
\end{theorem}

We point out in passing that there are two multiplicative sequences $f \colon \NN \to \CC$ satisfying $f(n) = 0$ for $n \geq 2$, specified by $f(1) = 0$ or $f(1) = 1$.

\begin{alphatheorem}\label{thm:intro:main-compl}
	Let $f \colon \NN \to \CC$ be a completely multiplicative generalised polynomial sequence. Then either there exists a Dirichlet character $\chi \colon \NN \to \CC$ and an exponent $a \in \NN_0$ such that $f(n) = \chi(n)n^a$ for all $n \in \NN$, or $f(n) = 0$ for all $n \in \NN$ with $n \geq 2$.
\end{alphatheorem}
\begin{proof}
	We know from Theorem \ref{thm:intro:main} that $f(n) = \chi(n) n^a$ for all $n \in \NN$, where $\chi$ is periodic and multiplicative, or $f(n) = 0$ for asymptotically almost all $n \in \NN$. In the former case, since $\chi(n) = f(n)/n^a$ for all $n \in \NN$, we see that $\chi$ is completely multiplicative, and hence it must be a Dirichlet character or identically zero.
	
	In the latter case, consider the sequence $g \colon \NN \to \{0,1\}$ given by $g(n) = 1$ if $f(n) \neq 0$ and $g(n) = 0$ if $f(n) = 0$. Then $g$ is completely multiplicative and it is not hard to show that it is a generalised polynomial (see e.g.\ \cite[Prop 5.4]{AdamczewskiKonieczny-2023-TAMS}). If there is at least one integer $k \geq 2$ with $f(k) \neq 0$ then for all $n \in \NN$ we have $g(kn) = g(n)$, which combined with Theorem \ref{thm:sparse} implies that $g(n) = 0$ and consequently also $f(n) = 0$ for all $n \in \NN$, leading to a contradiction. Thus, $f(n) = 0$ for all $n \geq 2$. (In fact, this step is also a special case of \cite[Prop. 12.16]{AdamczewskiKonieczny-2023-TAMS}, which is proved using a very similar argument.)
\end{proof}

\color{checked}
\subsection{Open problems}

As alluded to earlier, we note that, as partial converse to Theorem \ref{thm:intro:main}, for each periodic multiplicative sequence $\chi \colon \NN \to \CC$ and each exponent $a \in \NN_0$, the sequence $f$ given by $f(n) = \chi(n)n^a$ is both multiplicative and generalised polynomial. A question that is left open is the classification of multiplicative generalised polynomial sequences that are zero almost everywhere. In order to illustrate the difficulty of dealing with this problem, consider a rapidly increasing sequence of primes $(p_i)_{i=1}^\infty$ and let $f \colon \NN \to \{0,1\}$ be given by
\begin{equation}\label{eq:intro:def-FP}
	f(n) =
	\begin{cases}
		1 &\text{if } n = \prod_{i \in I} p_i \text{ for some finite } I \subset \NN;\\
		0 &\text{otherwise}.
	\end{cases}
\end{equation}
Evidently, thus defined sequence $f$ is multiplicative (but not completely multiplicative) and almost everywhere zero (assuming, for instance, that $\set{p_i}{i \geq 1}$ has zero relative density inside the set of all primes). However, it is not clear whether we should expect $f$ to be a generalised polynomial. For comparison, we point out that for each sequence of positive integers $(n_i)_{i=1}^\infty$ that increases rapidly enough, the corresponding characteristic sequence given by
\begin{equation}
	n \mapsto
	\begin{cases}
		1 &\text{if } n = n_i \text{ for some } i \in \NN;\\
		0 &\text{otherwise},
	\end{cases}
\end{equation}
is a generalised polynomial. (To be more precise, we require that there exists a constant $c > 0$ such that $\log n_{i+1} \geq (1+c) \log n_i$ for all sufficiently large $i$, meaning that $n_i$ has at least doubly exponential growth \cite{ByszewskiKonieczny-2018-TAMS}.) On the other hand, as a particular case of Theorem \ref{thm:intro:main-compl},
we see that the completely multiplicative sequence defined in analogy with \eqref{eq:intro:def-FP} by
\begin{equation}\label{eq:intro:def-FP-prime}
	n \mapsto
	\begin{cases}
		1 &\text{if } n = \prod_{i \in I} p_i^{\alpha_i} \text{ for some finite } I \subset \NN \text{ and } \a \in \NN^I;\\
		0 &\text{otherwise},
	\end{cases}
\end{equation}
is not generalised polynomial.

\subsection{Notation}

We let $\NN = \{1,2,\dots\}$ denote the set of positive integers and put $\NN_0 = \NN \cup \{0\}$. For $N \in \NN$ we let $[N] = \{0,1,2,\dots,N-1\}$ denote the length-$N$ initial interval of $\NN_0$. For $x \in \RR$, we let $\ip{x} \in \ZZ$ denote the integer part of $x$, sometimes also called the floor of $x$, and we let $\fp{x} = x - \ip{x} \in [0,1)$ denote the fractional part. We also let $\fpa{x} = \min( \fp{x}, 1-\fp{x})$ denote the circle norm of $x$.

We  use standard asymptotic notation, where $X = O(Y)$ or $X \ll Y$ if $X \leq C Y$ for an absolute constant $C > 0$. If the constant $C$ is allowed to depend on a parameter $p$, we write $X = O_p(Y)$ or $X \ll_p Y$. We also write $X = o_{p \to \infty}(Y)$ if $X/Y \to 0$ as $p \to \infty$.

\subsection*{Acknowledgements}
The author is grateful to Jakub Byszewski for extensive and fruitful discussions and for collaboration on related projects. The author works at the University of Oxford and is supported by UKRI Fellowship EP/X033813/1. For the purpose of open access, the author has applied a Creative Commons Attribution (CC BY) licence to any Author Accepted Manuscript version arising.

\color{black}

\section{Background}\color{checked}

\subsection{Generalised polynomials}
Generalised polynomial maps from $\ZZ$ to $\RR$, already informally defined in the introduction, are the smallest family $\GP$ of (two-sided) sequences that contains the usual polynomial sequences, constitutes a ring (under pointwise operations) and is closed under the operation of taking the integer part, meaning that if $g \in \GP$ then $\ip{g} \in \GP$, where $\ip{g}(n) := \ip{g(n)}$ for $n \in \ZZ$. By a one-sided generalised polynomial sequence we simply mean the restriction of a generalised polynomial on $\ZZ$ to $\NN$. By a generalised polynomial with values in $\RR^k$ ($k \in \NN$) we simply mean a $k$-tuple of $\RR$-valued generalised polynomials.

In subsequent discussion we will largely focus on bounded generalised polynomials. It is a standard observation that each generalised polynomial $g \colon \ZZ \to \RR$ admits an expansion of the form
\begin{align}\label{eq:bgrnd:exp-gp}
	g(n) &= g_d(n)n^d + \dots + g_1(n) n + g_0(n), & n \in \ZZ,
\end{align}
where $d \in \NN$ and $g_i \colon \ZZ \to \RR$ are bounded generalised polynomials. (This can be seen, for instance, by noting that the family of sequences of the form \eqref{eq:bgrnd:exp-gp} includes classical polynomials and is closed under addition, multiplication and taking the fractional part.) 

\subsection{Nilmanifolds}\label{sec:nilmanifolds}

We briefly discuss the prerequisites on nilmanifolds and their connection with generalised polynomials. For more details, we refer \cite{Tao-book} to introductory sections of papers such as \cite{MSTT,GreenTao-2012-Mobius,GreenTao-2012}; similar discussion can also be found in tangentially related \cite{KoniecznyMullner-2023+}.

Let $G$ be a $D$-dimensional $s$-step nilpotent Lie group and assume that $G$ is connected and simply connected. Let $\Gamma < G$ be a subgroup that is discrete and cocompact, meaning that the quotient space $G/\Gamma$ is compact. The space $G/\Gamma$ is called an \emph{$s$-step nilmanifold} and comes equipped with the Haar measure, which is the unique Borel probability measure on $G/\Gamma$ invariant under the action of $G$; we will use the notation $\int_{G/\Gamma} F(x) dx$ to denote the integral of a function $F \colon G/\Gamma \to \CC$ with respect to the Haar measure.

A \emph{degree-$d$ filtration} on $G$ is a sequence $G_\bullet$ of subgroups 
\[
	G = G_0 = G_1 \geq G_2 \geq G_3 \geq \dots
\] such that $G_{d+1} = \{e_G\}$ (and hence $G_{i} = \{e_G\}$ for all $i > d$) and for each $i,j$ we have $[G_i,G_j] \subset G_{i+j}$, where $[G_i,G_j]$ is the group generated by the commutators $[g,h] = ghg^{-1}h^{-1}$ with $g \in G_i$, $h \in G_j$. 

A \emph{Mal'cev basis} compatible with $\Gamma$ and $G_\bullet$ is a basis $\cX = (X_1,X_2,\dots,X_D)$ of the Lie algebra $\mathfrak{g}$ of $G$ such that
\begin{enumerate}
\item for each $0 \leq j \leq D$, the subspace $\fh_j := \operatorname{span}\bra{X_{j+1},X_{j+2},\dots,X_D}$ is a Lie algebra ideal in $\fg$;
 \item for each $0 \leq i \leq d$, each $g \in G_i$ has a unique representation as $g = \exp(t_{D(i) + 1} X_{t_{D(i) + 1}})  \cdots \exp(t_{D-1} X_{D-1}) \exp(t_D X_D) $, where $D(i) := \operatorname{codim} G_i$ and $t_j \in \RR$ for $D(i) < j \leq D$;
 \item $\Gamma$ is the set of all products $\exp(t_1 X_1) \exp(t_2 X_2) \cdots \exp(t_D X_D)$ with $t_j \in \ZZ$ for $1 \leq j \leq D$.
\end{enumerate}
We will usually keep the choice of the Mal'cev basis implicit, and assume that each filtered nilmanifold under consideration comes equipped with a fixed choice of Mal'cev basis. The Mal'cev basis $\cX$ induces bijective coordinate maps $\tau \colon G/\Gamma \to [0,1)^D$ and $\tilde \tau \colon G \to \RR^D$, such that
\begin{align*}
	x &= \exp(\tau_1(x)X_1) \exp(\tau_2(x)X_2) \cdots \exp(\tau_D(x)X_D) \Gamma, & x &\in G/\Gamma, \\
	g &= \exp(\tilde\tau_1(g)X_1) \exp(\tilde\tau_2(g)X_2) \cdots \exp(\tilde\tau_D(g)X_D), & g &\in G.
\end{align*}
The Mal'cev basis also induces a natural choice of a right-invariant metric on $G$ and a metric on $G/\Gamma$. We refer to \cite[Def.\ 2.2]{GreenTao-2012} for a precise definition. Keeping the dependence on $\cX$ implicit, we will use the symbol $d$ to denote either of those metrics.

A map $g \colon \ZZ \to G$ is polynomial with respect to the filtration $G_\bullet$, denoted $g \in \mathrm{poly}(\ZZ,G_\bullet)$, if it takes the form
\[
	g(n) = g_0^{} g_1^{n} \dots g_d^{\binom{n}{d}}, 
\]
where $g_i \in G_i$ for all $0 \leq i \leq d$ (cf.\ \cite[Lem.\ 6.7]{GreenTao-2012}). The restriction $g(an+b)$ of a polynomial $g$ to an arithmetic progression $a \ZZ + b$ is again a polynomial with respect to the same filtration.

\subsection{Representation of generalised polynomials}
A map $F \colon G/\Gamma \to \RR$ is \emph{piecewise polynomial} if, in Mal'cev coordinates, $G/\Gamma$ can be partitioned into a finite number of semialgebraic pieces such that the restriction of $F$ to each of them is polynomial, see \cite{BergelsonLeibman-2007} for details. For our purposes, we will only need to know that piecewise polynomial maps are continuous almost everywhere with respect to the Haar measure on $G/\Gamma$ and appear in the following representation theorem.

\begin{theorem}[{Corollary of \cite[Thm.\ A]{BergelsonLeibman-2007}}]\label{thm:BL}
	Let $f \colon \ZZ \to \RR^k$ be a bounded sequence. Then $f$ is a generalised polynomial if and only if there exists a connected, simply connected nilpotent Lie group $G$ of some dimension $D$, lattice $\Gamma < G$, a compatible filtration $G_\bullet$, a polynomial sequence $g \in \mathrm{poly}(\ZZ,G_\bullet)$ equidistributed modulo $\Gamma$, and a piecewise polynomial map $F \colon G/\Gamma \to \RR^k$ such that $f(n) = F(g(n)\Gamma)$ for all $n \in \ZZ$. 
\end{theorem}

 \color{normal}

\newcommand{\poly}{\operatorname{poly}}
\section{Equidistribution}

\color{checked} 

Following Green and Tao \cite{GreenTao-2012}, we will say that a finite sequence $(x_n)_{n=0}^{N-1}$ taking values in a nilmanifold $G/\Gamma$ is \emph{$\delta$-equidistributed} ($\delta > 0$) if for each Lipschitz function $F \colon G/\Gamma \to \RR$ we have 
\begin{equation}\label{eq:equi:def-d-equi}
	\abs{ \frac{1}{N} \sum_{n=0}^{N-1} F(x_n) - \int_{G/\Gamma} F(x) dx } \leq \delta \normLip{F},
\end{equation}
where the Lipschitz norm $\normLip{F}$ is defined as 
\begin{equation}\label{eq:equi:def-norm-Lip}
	\normLip{F} = \sup_{x} \abs{F(x)} + \sup_{x \neq y} \frac{\abs{F(x)-F(y)}}{d(x,y)}.
\end{equation}
As the name suggests, $\delta$-equidistribution is a quantitative refinement of the notion of equidistribution. Indeed, a sequence $(x_n)_{n=0}^\infty$ in $G/\Gamma$ is equidistributed if and only if for each $\delta > 0$ there exists $N_0$ such that for all $N \geq N_0$, the sequence $(x_n)_{n=0}^{N-1}$ is $\delta$-equidistributed. 

For a degree-$d$ polynomial sequence $f \colon \ZZ \to \RR$ with expansion
\begin{align}\label{eq:equi:46:1}
	f(n) &= a_d \binom{n}{d} + a_{d-1} \binom{n}{d-1} + \dots + a_{1} n + a_0,& n \in \ZZ,
\end{align}
and for an integer $N \in \NN$ we define the smoothness norm 
\begin{align}\label{eq:equi:def-C-infty}
	\norm{f}_{C^\infty[N]} = \max_{1 \leq i \leq d} N^i \fpa{a_i}. 
\end{align}
The expansion \eqref{eq:equi:46:1} has the advantage of being compatible with the discrete derivative; indeed, we have
\begin{align*}
f(n+1)-f(n) &= a_d \binom{n}{d-1} + a_{d-1} \binom{n}{d-2} + \dots + a_{0} n,& n \in \ZZ.
\end{align*}
As long as we do not need to worry about multiplicative factors dependent on $d$, we can equally well use the more typical expansion
\begin{align}\label{eq:equi:46:2}
	f(n) &= a_d' n^d + a_{d-1}' n^{d-1} + \dots + a_{1}' n + a_0',& n \in \ZZ,
\end{align}
and the corresponding smoothness norm
\begin{align}\label{eq:equi:def-C-infty-prime}
	\norm{f}_{C^\infty[N]}' = \max_{1 \leq i \leq d} N^i \fpa{a_i'}. 
\end{align}
Indeed, using the fact that the transition matrices used to pass from representation \eqref{eq:equi:46:1} to \eqref{eq:equi:46:2} or vice versa are upper triangular and have rational entries with denominators dividing $d!$ we see that
\begin{align}\label{eq:equi:smth-norm-equiv}
	\norm{d! f}_{C^\infty[N]}' &\ll_d \norm{f}_{C^\infty[N]}, & 
	\norm{d! f}_{C^\infty[N]} &\ll_d \norm{f}_{C^\infty[N]}'.	
\end{align}
The norm $\norm{d! f}_{C^\infty[N]}'$ behaves well when passing to arithmetic progressions.
\begin{lemma}\label{lem:equi:smth-AP}
	Let $f \colon \ZZ \to \RR$ be a degree-$d$ polynomial sequence, let $N \in \NN$, let $P = aM+b \subset [N]$ be an arithmetic progression of length $M = \mu N$, and let $\ell \colon [M] \to P$ be the parametrisation of $P$ given by $\ell(n) = an+b$. Then 
	\[ 
	\norm{a^d f}_{C^{\infty}[N]}' \ll_d (1/\mu^d)\norm{f \circ \ell}_{C^\infty[M]}'.
	\]
\end{lemma}
\begin{remark}
	With the same notation in Lemma \ref{lem:equi:smth-AP} above, we also have 	\(
	\norm{f}_{C^{\infty}[N]}' \gg_d \norm{f \circ \ell}_{C^\infty[M]}',
\)
although we will not need this estimate.
\end{remark}
\begin{proof}	
By definition, we have
	\begin{align*}
		(f \circ \ell)(m) &= \sum_{i=0}^d b_i' m^i, & m \in \ZZ,
	\end{align*}
	where $\abs{b_i'} \leq \norm{f \circ \ell}_{C^\infty[M]}'/M^i$. Thus, noting that $\ell^{-1}(n) = (n-b)/a$, we have
	\begin{align*}
		a^d f(n) &= a^d (f \circ \ell \circ \ell^{-1}) (n) = \sum_{i=0}^d b_i' a^{d-i} (n-b)^i
		\\ &= 
		\sum_{i=0}^d  \sum_{j=0}^i \binom{i}{j} b_i' a^{d-i}(-b)^{i-j}  n^j =: \sum_{j=0}^d c_j' n^j, & n \in \ZZ.
	\end{align*}
	Recalling the aforementioned estimate on $\fpa{b_i'}$ and noting that $a \leq 1/\mu$ and $b \leq N$ we conclude that for each $0 \leq j \leq d$ we have
	\begin{align*}
		\fpa{ c_j' } \ll_d \frac{ \norm{f \circ \ell}_{C^\infty[M]}'}{(\mu N)^i}  \cdot \frac{1}{\mu^{d-i}} \cdot N^{i-j} 
		\ll_d  \frac{1}{\mu^d} \cdot \frac{\norm{f \circ \ell}_{C^\infty[M]}'}{N^j},
	\end{align*}
	as needed.	
\end{proof}

\color{checked}

A horizontal character is a continuous homomorphism $\eta \colon G \to \RR/\ZZ$ which anihilates $\Gamma$ and hence (by a slight abuse of notation) can identified with a map $G/\Gamma \to \RR/\ZZ$. Mal'cev coordinates give a natural way to identify a horizontal character $\eta$ with an integer vector $k$, and we let $\abs{\eta} = \abs{k}$ denote the (supremum) norm. See \cite[Def.\ 2.6]{GreenTao-2012} for details.

We have now introduced all terminology needed to state (in a slightly simplified form) the main Theorem of \cite{GreenTao-2012}.

\begin{theorem}\label{thm:equi:GT}
	Fix a filtered nilmanifold $(G,G_\bullet,\Gamma)$. Let $\delta \in (0,1/2)$, $N \in \NN$, and $g \in \poly(G_\bullet)$. If the sequence $(g(n)\Gamma)_{n=0}^{N-1}$ is not $\delta$-equidistributed in $G/\Gamma$ then there exists a horizontal character $\eta \colon G/\Gamma \to \RR/\ZZ$ with $0 < \abs{\eta} \leq 1/\delta^{O(1)}$ such that $\norm{\eta \circ g}_{C^\infty[N]} \leq 1/\delta^{O(1)}$. (The implicit constants depend on $(G,G_\bullet,\Gamma)$.)
\end{theorem}

In the applications we have in mind, we will need a slightly different notion related to uniform distribution of a sequence. We will tentatively say that a sequence $(x_n)_{n=0}^{N-1}$ is \emph{$\rho$-dense} if for each arithmetic progression $P \subset [N]$ with $\abs{P} \geq \rho N$ and each radius-$\rho$ ball $B \subset G/\Gamma$ there is at least one $n \in P$ with $g(n)\Gamma \in B$. We have the following, fairly standard, corollary of Theorem \ref{thm:equi:GT}.
\begin{proposition}\label{prop:equi:rho-dense}
	Fix a filtered nilmanifold $(G,G_\bullet,\Gamma)$. Let $\rho \in (0,1/2)$, $N \in \NN$, and $g \in \poly(G_\bullet)$. If the sequence $(g(n)\Gamma)_{n=0}^{N-1}$ is not $\rho$-dense in $G/\Gamma$ then there exists a horizontal character $\eta \colon G/\Gamma \to \RR/\ZZ$ with $0 < \abs{\eta} \leq 1/\rho^{O(1)}$ such that $\norm{\eta \circ g}_{C^\infty[N]} \leq 1/\rho^{O(1)}$. (The implicit constants depend on $(G,G_\bullet,\Gamma)$.)
\end{proposition}
\begin{proof}
	Since the sequence $(g(n)\Gamma)_{n=0}^{N-1}$ is not $\rho$-dense, we can find an arithmetic progression $P = a[M]+b\subset [N]$ with $\abs{P} = M \geq \rho N$ and a radius-$\rho$ ball $\Ball(z,\rho) \subset G/\Gamma$ such that $g(n)\Gamma \not \in \Ball(z,\rho)$ for all $n \in P$. Let $\ell \colon [M] \to P$ be the parametrisation of $P$ given by $\ell(n) = an+b$. It is routine to construct $F \colon G/\Gamma \to [0,1]$ such that $F(x) = 0$ if $d(x,z) \geq 1$, $\normLip{F} \ll 1/\rho$ and $\int_{G/\Gamma} F(x) dx \gg \rho^{\dim G}$ (for instance, one can take $F(x) = \max(0,1-d(x,z)/\rho)$).

	Suppose that for some $\delta > 0$, the sequence $(g(\ell(n))\Gamma)_{n=0}^{M-1}$ is $\delta$-equidistributed. Then
	\begin{align}\label{eq:equi:77:1}
		\int_{G/\Gamma} F(x) dx &=
		\abs{ \sum_{n=0}^{M-1} F(g(\ell(n))\Gamma - \int_{G/\Gamma} F(x) dx} \leq \delta \normLip{F}.
	\end{align}
	Thus, we can find $\delta \gg \rho^{\dim G+2}$ such that $(g(\ell(n))\Gamma)_{n=0}^{M-1}$ is not $\delta$-equidistributed. By Theorem \ref{thm:equi:GT}, we can find a horizontal character $\eta \colon G/\Gamma \to \RR/\ZZ$ with $0 < \abs{\eta} \leq 1/\delta^{O(1)} \leq 1/\rho^{O(1)}$ such that  $\norm{\eta \circ g \circ \ell}_{C^\infty[M]} \leq 1/\delta^{O(1)} \leq 1/\rho^{O(1)}$. Applying Lemma \ref{lem:equi:smth-AP} and replacing $\eta$ with $(d!)^2 a^d \eta$ if necessary, we conclude that $\norm{\eta \circ g}_{C^\infty[N]} \leq 1/\rho^{O(1)}$, as needed.
\end{proof}

Another ingredient which we will need is the following variant of Weyl's theorem along the primes, due to Rhin \cite{Rhin-1973}. We let $\cP = \{2,3,5,\dots,\}$ denote the set of primes. We will say that a sequence $(x_p)_{p \in \cP}$ is equidistributed in some space $\Omega$ if $(x_{p_n})_{n=0}^\infty$ is equidistributed in $\Omega$, where $(p_n)_{n=0}^\infty$ is the increasing enumeration of $\cP$.

\begin{theorem}\label{thm:equi:Rh}
	Let $f \colon \ZZ \to \RR$ be a polynomial sequence. Suppose that $f(x) = \sum_{i=0}^d a_i x^d$ ($x \in \RR$) and that $a_i$ is irrational for at least one $1 \leq i \leq d$. Then the sequence $(f(p) \bmod 1)$ is equidistributed in $\RR/\ZZ$.
\end{theorem}

In particular, bearing in mind the prime number theorem, we conclude that for each polynomial sequence $f \colon \ZZ \to \RR$ with at least one irrational coefficient other than the constant term and for each interval $I \subset [0,1)$, the number of primes $p \in [X,2X)$ such that $\fp{f(p)} \in I$ is $( \abs{I} + o_{X\to\infty}(1))X/\log X$.

\begin{remark}
	The behaviour of more complicated generalised polynomials along the primes was further investigated in \cite{GreenTao-2012-Mobius} and \cite{BergelsonHalandSon-2021} (see also \cite{BKS-2019}). However, we will not make use of these results.
\end{remark}

The following technical result, obtained by combining Theorem \ref{thm:equi:GT} and Theorem \ref{thm:equi:Rh}, is at the heart of our argument.

\begin{proposition}
	Fix a filtered nilmanifold $(G,G_\bullet,\Gamma)$ and a polynomial sequence $g \in \poly(G_\bullet)$ such that $(g(n)\Gamma)_n$ is equidistributed in $G/\Gamma$. Then for each $\delta \in (0,1/2)$ there exists $N_0 > 0$ such that for all $N \geq N_0$ there exists $X_0 > 0$ such that for all $X \geq X_0$ there are   $\brabig{1-O_\delta(1/N)}X/{\log X}$ primes $p \in [X,2X)$ such that the sequence $\brabig{\bra{g(n)\Gamma, g(pn)\Gamma}}_{n=0}^{N-1}$ is $\delta$-dense in $G/\Gamma$.
\end{proposition}
\begin{proof}
	Fix $\delta > 0$; throughout the argument we let all implicit constants depend on $\delta$, as well as on $(G,G_\bullet,\Gamma)$ and $g$. For a prime $p$, let $g_p \colon \ZZ \to G$ and $\tilde g_p \colon \ZZ \to G \times G$ be the polynomial sequences given by $g_p(n) = g(pn)$ and $\tilde g_p(n) = (g(n),g_p(n))$. We assume that $N$ is sufficiently large in terms of $\delta$ and $(G,G_\bullet,\Gamma)$ and that $X$ is sufficiently large in terms of $N$. 
	
	Consider a prime $p \in [X,2X)$ such that the sequence $\brabig{\bra{g(n)\Gamma, g(pn)\Gamma}}_{n=0}^{N-1}$ is not $\delta$-dense. It follows from Theorem \ref{prop:equi:rho-dense} that there exists a horizontal character $\eta \colon (G/\Gamma)\times(G/\Gamma) \to \RR/\ZZ$ with $0 \neq \abs{\eta} = O(1)$ such that $\norm{\eta \circ \tilde g_p}_{C^\infty[N]} = O(1)$. Since the prime number theorem implies that there are $(1+o_{X\to\infty}(1))X/\log X$ primes in $[X,2X)$ and since the number of possible choices of $\eta$ is $O(1)$, it will suffice to show that for each choice of a horizontal character $\eta$ and each constant $C > 0$ there are $O\bra{ N^{-1} X/\log X}$ primes $p$ such that $\norm{\eta \circ \tilde g_p}_{C^\infty[N]} \leq C$. For brevity, let us refer to such primes as \emph{bad}.
	
We can write $\eta$ in the form $\eta(x,y) = \kappa(x) + \lambda(y)$ where $\kappa,\lambda \colon G/\Gamma \to \RR/\ZZ$ are horizontal characters with $\abs{\kappa},\abs{\lambda} = O(1)$. Hence, $\eta \circ \tilde g_p = \kappa \circ g + \lambda \circ g_p$. Writing 
\begin{align}
	\kappa \circ g(n) &= \sum_{i=0}^d b_i n^i, &
	\lambda \circ g(n) &= \sum_{i=0}^d a_i n^i, & n \in \ZZ,
\end{align}
we have $\norm{\eta \circ \tilde g_p}_{C^\infty[N]} = \max_{1 \leq i \leq d} N^i\fpa{p^i a_i + b_i}$. Since $\bra{g(n)\Gamma}_{n}$ is equidistributed, at least one of the coefficients $a_i$ ($1 \leq i \leq d$) is irrational, unless $\lambda = 0$. The same remark applies to the coefficients $b_i$ and the character $\kappa$.

If $\lambda = 0$ then $\kappa \neq 0$ and we have
\begin{align*}
\norm{\eta \circ \tilde g_p}_{C^\infty[N]} = \norm{\kappa \circ g}_{C^\infty[N]} = \max_{1 \leq i \leq d} N^i \fpa{b_i}.
\end{align*}
If $N \geq C/\max_{i \leq i \leq d} \fpa{b_i}$ then there are not bad primes; otherwise we bound their number trivially by $2X/\log X$. In either case, the number of bad primes is bounded by $(2C/\max_{i \leq i \leq d} \fpa{b_i}) N^{-1} X/\log X$, as needed.

Suppose next that $\lambda \neq 0$. Pick any $1 \leq i \leq d$ such that $a_i$ is irrational. For each bad prime we have, in particular, $\fpa{p^i a_i + b_i} \leq C/N^i$. Hence, it follows from Theorem \ref{thm:equi:Rh} that for sufficiently large $X$, the number of bad primes is bounded by $4C N^{-i} X/\log X \ll N^{-1} X/\log X$, as needed.
\end{proof}

\color{black} \color{normal}

\section{Proof of the Main Theorem}\color{checked}

We have now collected all the ingredients needed to prove Theorem \ref{thm:intro:main}.

Recall that we can write $f$ is the form \eqref{eq:bgrnd:exp-gp}:
\begin{align}\label{eq:proof:exp-gp}
	f(n) &= f_d(n)n^d + \dots + f_1(n) n + f_0(n), & n \in \ZZ,
\end{align}
with $d \in \NN$ and $f_i \colon \ZZ \to \RR$ bounded. By Theorem \ref{thm:BL} we can find a filtered nilmanifold $(G,G_\bullet, \Gamma)$, a sequence $g \in \poly(G_\bullet)$ equidistributed modulo $\Gamma$, and piecewise polynomial maps $F_{i,r} \colon G/\Gamma \to \CC$ ($0 \leq i \leq d$, $0 \leq r < Q$) such that
\begin{align}\label{eq:proof:f=F}
	f_i(n) &= F_{i,r}(g(n)\Gamma), & n \equiv r \bmod Q.
\end{align}

We claim that for each $0 \leq r < Q$ either $F_{i,r} = 0$ a.e.\ for all $0 \leq i \leq d$, or there exists exactly one $0 \leq d_r' \leq d$ and $\lambda_r \in \CC \setminus \{0\}$ such that $F_{i,r} = \lambda_r$ a.e.\ for $i = d_r'$ and $F_{i,r} = 0$ a.e.\ for $i \neq d_r'$. We argue for each $r$ independently. If $F_{i,r} = 0$ a.e.\ for all $0 \leq i \leq d$ then we are done, so suppose that this is not the case. For the sake of notational clarity, we assume that $F_{d,r}$ is not a.e.\ zero, meaning that we aim to prove the claim with $d_{r}' = d$. In the general case we define $d_r'$ to be the largest index $i$ with $F_{i,r}$ not a.e.\ zero and argue similarly. We proceed by induction on $0 \leq i \leq d$, in descending order.

Fix $\delta > 0$. By Proposition \ref{prop:equi:rho-dense} (and the prime number theorem in arithmetic progressions), for each sufficiently large $N$ and for each $X$ sufficiently large in terms of $N$ we can find $p \in [X,2X)$ such that $p \equiv 1 \bmod Q$ and the sequence $\brabig{g(n)\Gamma,g(pn)\Gamma}_{n=0}^{N-1}$ is $\delta$-dense. 

Pick any $x_1,y \in G/\Gamma$ and $t_1 \in [1/2,1]$. Because of $\delta$-density, we can find an integer $n_1$ such that 
\begin{align}\label{eq:proof:def-n-1}
	n_1 &\equiv r \bmod Q, &
	\abs{\frac{n_1}{N} - t_1} &\leq \delta,&
	d\brabig{ g(n_1)\Gamma, x_1} &\leq Q\delta,&
	d\brabig{ g(pn_1)\Gamma, y} &\leq \delta.	
\end{align}
By the same token, for any $x_2 \in G/\Gamma$ and $t_2 \in [1/2,1]$ we can find an integer $n_2$ such that
\begin{align}\label{eq:proof:def-n-2}
	n_2 &\equiv r \bmod Q, &
	\abs{\frac{n_2}{N} - t_2} &\leq \delta,&
	d\brabig{ g(n_2)\Gamma, x_2} &\leq Q\delta,&
	d\brabig{ g(pn_2)\Gamma, y} &\leq \delta.	
\end{align}
Because $f$ is multiplicative, we have (assuming that $X > N$, as we may):
\begin{align}\label{eq:proof:ratio-1}
	f(n_1)f(pn_2) = f(n_2) f(pn_1).
\end{align}
Using \eqref{eq:proof:exp-gp}, \eqref{eq:proof:def-n-1} and \eqref{eq:proof:def-n-2} to approximate the factors appearing in \eqref{eq:proof:ratio-1}, and assuming that $F_{d,r}$ is continuous at $x_1,x_2$ and $y$ and that $N$ is sufficiently large as a function of $\delta$, we obtain
\begin{align*}
	&\brabig{F_{d,r}(x_1) + O(\delta)} n_1^d \cdot \brabig{F_{d,r}(y) + O(\delta)} (pn_2)^d =\\
	&\brabig{F_{d,r}(x_2) + O(\delta)} n_2^d \cdot \brabig{F_{d,r}(y) + O(\delta)} (pn_1)^d,
\end{align*}
which after elementary manipulations implies that
\begin{align}\label{eq:proof:F_d,r=const}
\abs{F_{d,r}(x_1) - F_{d,r}(x_2)} \cdot \abs{F_{d,r}(y)} = O(\delta).
\end{align}
Holding $x_1,x_2,y$ constant and passing to the limit $\delta \to 0$, we conclude that $F_{d,r}(x_1) = F_{d,r}(x_2)$ provided that $F_{d,r}$ is continuous at $x_1, x_2$ and $y$ (which holds for almost all $x_1,x_2$ and $y$) and $F_{d,r}(y) \neq 0$ (which we can ensure thanks to the assumption that $F_{d,r}$ is not a.e.\ zero). Thus, there is some $\lambda_r \neq 0$ such that $F_{d,r} = \lambda_r$ a.e..

Suppose next that we want to prove the claim for some $0 \leq i < d$ and that we have already dealt with all $i < j \leq d$. We proceed just like above, except that we use $d-i$ leading terms of \eqref{eq:proof:exp-gp} to approximate the factors in \eqref{eq:proof:ratio-1}. Assuming that $F_{j,r}$ are continuous at $x_1,x_2$ and $y$ for $i \leq j \leq d$, we obtain:
\begin{align*}
	&\bra{\lambda_r n_1^d + \brabig{F_{i,r}(x_1)+ O(\delta)} n_1^i} \cdot \brabig{\lambda_r+O(1/X)} (pn_2)^d =\\
	&\bra{\lambda_r n_2^d + \brabig{F_{i,r}(x_2)+ O(\delta)} n_2^i} \cdot \brabig{\lambda_r+O(1/X)} (pn_1)^d,
\end{align*}
which after elementary manipulations implies that
\begin{align}\label{eq:proof:F_d,r=const-2}
\abs{F_{i,r}(x_1)t_2^{d-i} - F_{i,r}(x_2)t_1^{d-i}} = O(\delta).
\end{align}
If $F_{i,r}$ was not a.e.\ zero then we could find a point $x$ satisfying the continuity assumptions mentioned above and such that $F_{i,r}(x) \neq 0$. Then, putting $x_1 = x_2 = y$ and $t_1 \neq t_2$ and passing to the limit $\delta \to 0$ we would obtain a contradiction. Thus $F_{i,r} = 0$ a.e., as needed.
 
We have shown that for asymptotically almost all $n$, letting $r = n \bmod Q$, we have
\begin{align}\label{eq:proof:f=l_r*n^d_r}
	f(n) &= \lambda_r n^{d_r}.
\end{align}
Pick $0 \leq r,s < Q$ such that $\gcd(r,s,Q) = 1$. Next, pick $n \equiv r \bmod Q$ such that \eqref{eq:proof:f=l_r*n^d_r} holds for $n$, and then pick $m \equiv s$ such that $\gcd(m,n) = 1$ and \eqref{eq:proof:f=l_r*n^d_r} holds for $nm$ and $m$ (asymptotically almost all $n$, and asymptotically almost $m$ in the relevant residue classes are suitable). We then have
\begin{align}\label{eq:proof:multi}
	\lambda_r\lambda_s n^{d_r} m^{d_s} = f(n)f(m) = f(nm) = \lambda_{rs} (nm)^{d_{rs}}. 
\end{align}
Comparing the two sides of \eqref{eq:proof:multi} and setting $s = 1$ (meaning that $r$ can be arbitrary) we see that $d_r = d_1$ for all $r$ such that $\lambda_r \neq 0$. Since the values of $d_r$ for $r$ with $\lambda_r = 0$ are irrelevant, we may freely assume that $d_r$ takes the same value for all $0 \leq r < Q$, say $d_r = d$. Now, \eqref{eq:proof:multi} simplifies to
\begin{align}\label{eq:proof:multi-2}
	\lambda_r\lambda_s = \lambda_{rs}, 
\end{align}
meaning that the sequence $\chi \colon \NN \to \CC$ defined by $\chi(n) = \lambda_r$ for $n \equiv r \bmod Q$ is multiplicative.

Since $\chi$ is not identically zero, we have $\chi(1) = 1$. Pick any $n \in \NN$. Reasoning like above, for asymptotically almost all $m \equiv 1 \bmod nQ$ we have
\begin{align}\label{eq:proof:multi-3}
	f(n) = \frac{f(nm)}{f(m)} = \frac{\chi(nm)(nm)^d}{\chi(m)n^d} = \chi(n)n^d.
\end{align}
Thus, we have $f(n) = \chi(n)n^d$ for all $n \in \NN$, as needed. \color{normal}

\bibliographystyle{alphaabbr}
\bibliography{bibliography}
\end{document}